\documentclass[11pt,reqno]{amsart}

\usepackage{amssymb}

\usepackage{amsfonts}

\usepackage{amsmath}

\usepackage[all]{xy}

\newtheorem{theorem}{Theorem}[section]

\newtheorem{proposition}[theorem]{Proposition}

\newtheorem{Definition}[theorem]{Definition}

\newtheorem{Example}[theorem]{Example}

\newtheorem{Remark}[theorem]{Remark}

\newenvironment{remark}{\begin{Remark}\begin{em}}{\end{em}\end{Remark}}

\newenvironment{definition}{\begin{Definition}\begin{em}}{\end{em}\end{Definition}}

\address{Sejong Kim \\ Department of Mathematics, Chungbuk National University, Cheongju 361-763, Korea}
\email{skim@chungbuk.ac.kr}

\begin{document}

\author{Sejong Kim}

\title[Mixture and interpolation of the parameterized ordered means]{Mixture and interpolation of the parameterized ordered means}

\date{}
\maketitle

\begin{abstract}
Loewner partial order plays a very important role in metric topology and operator inequality on the open convex cone of positive invertible operators. In this paper we consider a family $G = \{ G_{n} \}_{n \in \mathbb{N}}$ of the ordered means for positive invertible operators equipped with homogeneity and properties related to the Loewner partial order such as the monotonicity, joint concavity, and arithmetic-$G$-harmonic weighted mean inequalities. Similar to the resolvent average, we construct a parameterized ordered mean and compare two types of the mixture of parameterized ordered means in terms of the Loewner order. We also show the relation between two families of parameterized ordered means associated with the power mean, monotonically interpolating given two parameterized ordered means.

\vspace{5mm}


\end{abstract}

\section{Introduction}

As the proximal average of proper convex lower semi-continuous functions in the context of convex analysis and optimization, the weighted resolvent mean which is a parameterized harmonic mean has been introduced \cite{BMW}:
\begin{displaymath}
\mathcal{R}^{\mu} (\omega; \mathbf{A}) := \left[ \sum_{i=1}^{n} w_{i} (A_{i} + \mu I)^{-1} \right]^{-1} - \mu I, \ \mu \geq 0,
\end{displaymath}
where $\omega = (w_{1}, \dots, w_{n})$ is a positive probability vector in $\mathbb{R}^{n}$ and $\mathbf{A} = (A_{1}, \dots, A_{n})$ is an $n$-tuple of positive definite Hermitian matrices. As a symmetrized version of the weighted resolvent mean and a unique minimizer of the weighted sum of Kullback-Leibler divergence, a parameterized weighted arithmetic-geometric-harmonic mean (simply call it the weighted $\mathcal{A} \# \mathcal{H}$ mean) has been introduced \cite{KLL}:
\begin{displaymath}
\mathcal{L}^{\mu} (\omega; \mathbf{A}) := \left[ \sum_{i=1}^{n} w_{i} (A_{i} + \mu I) \right] \# \left[ \sum_{i=1}^{n} w_{i} (A_{i} + \mu I)^{-1} \right]^{-1} - \mu I, \ \mu \geq 0,
\end{displaymath}
where $A \# B$ is the midpoint of Riemannian geodesic $A \#_{p} B = A^{1/2} (A^{-1/2} B A^{-1/2})^{p} A^{1/2}, p \in [0,1]$ of positive definite Hermitian matrices $A$ and $B$ for the Riemannian trace metric $\delta(A, B) = \Vert \log A^{-1/2} B A^{-1/2} \Vert_{2}$. The interesting results of these means are that they interpolate the weighted harmonic mean and arithmetic mean, the weighted $\mathcal{A} \# \mathcal{H}$ mean is the limit of the mean iteration of two-variable arithmetic mean and resolvent mean, and they satisfy the monotonicity for parameter $\mu$ and the non-expansiveness for the Thompson part metric $d_{T}(A, B) = \Vert \log A^{-1/2} B A^{-1/2} \Vert$, where $\Vert \cdot \Vert$ denotes the operator norm. Recently, the generalization of parameterized version of weighted means including the Cartan mean, which is the unique minimizer of the weighted sum of Riemannian trace distances to given variables, to contractive barycentric maps of probability measures has been developed \cite{Lim}.

On the open convex cone of positive invertible (positive definite) bounded linear operators as the infinite-dimensional setting, we consider a family $G = \{ G_{n} \}$ of the $n$-variable weighted means equipped with homogeneity and properties related to the Loewner partial order for each $n \in \mathbb{N}$, that is, the monotonicity, joint concavity, and arithmetic-$G$-harmonic mean inequalities. It includes many multivariate means such as the resolvent mean, power mean, Karcher mean \cite{LL14}, and we call it the \emph{ordered mean}. Similar to the weighted resolvent mean and the weighted $\mathcal{A} \# \mathcal{H}$ mean, we construct the parameterized ordered mean from given ordered mean $G$:
\begin{displaymath}
G^{\mu} (\omega; \mathbf{A}) := G^{\mu} (\omega; A_{1} + \mu I, \dots, A_{n} + \mu I) - \mu I, \ \mu \geq 0,
\end{displaymath}
and $G^{\mu} (\omega; \mathbf{A}) := G^{-\mu} (\omega; \mathbf{A}^{-1})^{-1}$ for $\mu < 0$, where $\mathbf{A}^{-1} := (A_{1}^{-1}, \dots, A_{n}^{-1})$. We first in Section 3 investigate properties of the parameterized ordered mean additionally to those in \cite{Lim}, and then compare two mixed means of parameterized ordered means: for the $n$-by-$k$ block matrix $\mathbb{A} = [ A_{ij} ]$ whose block entries are positive definite operators, and a positive probability vector $\lambda \in \mathbb{R}^{k}$,
\begin{displaymath}
G_{n}^{\nu}(\omega; G_{k}^{\mu_{1}}(\lambda; \mathbb{A}^{1}), \dots, G_{k}^{\mu_{n}}(\lambda; \mathbb{A}^{n})) \ \textrm{ and } \ G_{k}^{\sum \omega_{i} \mu_{i}} (\lambda; G_{n}^{\nu}(\omega; \mathbb{A}_{1}), \dots, G_{n}^{\nu}(\omega; \mathbb{A}_{k})),
\end{displaymath}
where $\mathbb{A}^{i}$ and $\mathbb{A}_{j}$ denote the tuples of the $i$th row and $j$th column of $\mathbb{A}$. They coincides when the variables $A_{ij}$ commute, but it does not hold in general. We obtain interesting inequalities associated with the Kantorovich constant.

Furthermore, we consider in Section 4 two families of parameterized ordered means
\begin{displaymath}
\{ G^{P_{p} (1-t, t; \mu, \nu)} (\omega; \mathbf{A}) \}_{t \in [0,1]} \ \textrm{ and } \ \{ P_{p} (1-t, t; G^{\mu} (\omega; \mathbf{A}), G^{\nu} (\omega; \mathbf{A})) \}_{t \in [0,1]}
\end{displaymath}
for given parameters $\mu, \nu > 0$ and any $p \in [0,1]$. Note that $P_{p} (\omega; \mathbf{A})$ is the weighted power mean of positive definite operators, which is the unique positive definite solution $X$ of the nonlinear equation $X = \sum_{i=1}^{n} w_{i} X \#_{p} A_{i}$. The interesting fact of these families is that they interpolate monotonically two parameterized ordered means $G^{\mu} (\omega; \mathbf{A})$ and $G^{\nu} (\omega; \mathbf{A})$, due to the monotonicities of parameterized ordered means on parameters and power means on variables. We show their relation with respect to the Loewner order and provide a generalization to the multivariate power means, so that we obtain the interesting chain of inequalities for positive parameters. Finally, we give in Section 5 two open problems about the interpolation of parameterized ordered means for the generalized means (H\"{o}lder means) instead of the power means and the contractive barycentric maps of probability measures with compact support.

\section{Ordered means}

Let $B(\mathcal{H})$ be the Banach space of all bounded linear operators on a Hilbert space $\mathcal{H}$ with inner product $\langle \cdot, \cdot \rangle$, and let $S(\mathcal{H}) \subset B(\mathcal{H})$ be the real vector space of all self-adjoint operators. We call $A \in S(\mathcal{H})$ positive semi-definite (positive definite) if $\langle x, Ax \rangle \geq (>) 0$ for all (nonzero, respectively) vector $x \in \mathcal{H}$. We denote as $\mathbb{P} \subset S(\mathcal{H})$ the open convex cone of all positive definite operators. For self-adjoint operators $A, B$ we write $A \leq (<) B$ if and only if $B - A$ is positive semi-definite (positive definite, respectively). This is known as the Loewner partial order.

Since Kubo and Ando \cite{KA} established two-variable means of positive definite matrices and operators, many different kinds of construction schemes of $n$-variable means have been developed. Especially, Ando, Li, and Mathias \cite{ALM} suggested ten desired properties for extended geometric means. We consider a family of the weighted means of positive definite operators with homogeneity and properties only related with the Loewner order, and call it the ordered mean. In the following $\Delta_{n}$ is the simplex of positive probability vectors in $\mathbb{R}^{n}$ convexly spanned by the unit coordinate vectors.

\begin{definition}
The \emph{ordered mean} is a family $G = \{ G_{n} \}_{n \in \mathbb{N}}$ such that for each $n$, a map $G_{n}: \Delta_{n} \times \mathbb{P}^{n} \to \mathbb{P}$ satisfies the following properties: for $\mathbf{A} = (A_{1}, \dots, A_{n}), \ \mathbf{B} = (B_{1}, \dots, B_{n}) \in \mathbb{P}^{n}, \ \omega = (w_{1}, \dots, w_{n}) \in \Delta_{n}$, and a positive real number $a$,
\begin{itemize}
\item[(P1)] (Homogeneity) $\displaystyle G_{n}(\omega; a \mathbf{A}) = a G_{n}(\omega; \mathbf{A})$;
\item[(P2)] (Monotonicity) $\displaystyle G_{n}(\omega; \mathbf{B}) \leq G_{n}(\omega; \mathbf{A})$ whenever $B_{i} \leq A_{i}$ for all $1 \leq i \leq n$;
\item[(P3)] (Joint concavity) $\displaystyle G_{n}(\omega; (1 - s) \mathbf{A} + s \mathbf{B}) \geq (1 - s) G_{n}(\omega; \mathbf{A}) + s G_{n}(\omega; \mathbf{B})$ for $0 \leq s \leq 1$;
\item[(P4)] (Arithmetic-$G$-harmonic weighted mean inequalities)
\begin{displaymath}
\mathcal{H}(\omega; \mathbf{A}) := \left[ \sum_{i=1}^{n} w_{i} A_{i}^{-1} \right]^{-1} \leq G_{n}(\omega; \mathbf{A}) \leq \sum_{i=1}^{n} w_{i} A_{i} =: \mathcal{A}(\omega; \mathbf{A}).
\end{displaymath}
\end{itemize}
\end{definition}

By the arithmetic-$G$-harmonic weighted mean inequalities (P4), one can see that the ordered mean $G$ is idempotent, that is, $G_{n}(\omega; A, \dots, A) = A$ for any $A \in \mathbb{P}$ and all $n \in \mathbb{N}$.

\begin{remark}
Many multivariate means of positive definite matrices and operators, including the Ando-Li-Mathias mean \cite{ALM}, Bini-Meini-Poloni mean \cite{BMP}, resolvent average \cite{BMW}, arithmetic-geometric-harmonic mean \cite{KLL}, power mean \cite{LP}, and Karcher mean \cite{LL14}, fulfill the definition of ordered means. Moreover, every ordered mean $G = \{ G_{n} \}_{n \in \mathbb{N}}$ is the Lie-Trotter mean \cite{HK} since it satisfies the arithmetic-$G$-harmonic weighted mean inequalities. That is, for each $n$
\begin{displaymath}
\lim_{s \to 0} G_{n}(\omega; \gamma_{1}(s), \dots, \gamma_{n}(s))^{1/s} = \exp \left[ \sum_{i=1}^{n} w_{j} \gamma_{i}'(0) \right],
\end{displaymath}
where $\omega \in \Delta_{n}$ and $\gamma_{1}, \dots, \gamma_{n}$ are differentiable curves on $\mathbb{P}$ with $\gamma_{i}(0) = I$ for all $i$.
\end{remark}

\begin{remark}
From \cite[Proposition 2.3]{KumL} the ordered mean $G$ is non-expansive for the Thompson metric $d_{T}$. In other words, let $\mathbf{A} = (A_{1}, \dots, A_{n}), \ \mathbf{B} = (B_{1}, \dots, B_{n}) \in \mathbb{P}^{n}$, and $\omega = (w_{1}, \dots, w_{n}) \in \Delta_{n}$. For each $n \in \mathbb{N}$
\begin{displaymath}
d_{T} (G_{n}(\omega; \mathbf{A}), G_{n}(\omega; \mathbf{B})) \leq \underset{1 \leq i \leq n}{\max} d_{T} (A_{i}, B_{i}),
\end{displaymath}
where $d_{T} (A, B) = \Vert \log A^{-1/2} B A^{-1/2} \Vert$ for any $A, B \in \mathbb{P}$ and the operator norm $\Vert \cdot \Vert$. This provides a generalization to the contractive barycentric map of probability measures \cite{LL17, Lim}, as well as continuity of the ordered mean $G_{n}$.
\end{remark}

Let
\begin{displaymath}
\mathbb{A} = [ A_{ij} ] =
\left(
  \begin{array}{cccc}
    A_{11} & A_{12} & \cdots & A_{1k} \\
    A_{21} & A_{22} & \cdots & A_{2k} \\
    \vdots & \vdots & \ddots & \vdots \\
    A_{n1} & A_{n2} & \cdots & A_{nk} \\
  \end{array}
\right)
\end{displaymath}
be an $n$-by-$k$ block matrix whose $(i,j)$ block entry is $A_{ij} \in \mathbb{P}$. We denote as $\mathbb{A}^{i} := (A_{i1}, A_{i2}, \dots, A_{ik}) \in \mathbb{P}^{k}$ and $\mathbb{A}_{j} := (A_{1j}, A_{2j}, \dots, A_{nj}) \in \mathbb{P}^{n}$, respectively, the tuples of the $i$th row and $j$th column of $\mathbb{A}$. Also, we simply write as $A_{1} \oplus A_{2} \oplus \cdots \oplus A_{n}$ the $n$-by-$n$ block diagonal matrix whose $(i,i)$ block entry is $A_{i} \in \mathbb{P}$.

Given $\omega = (w_{1}, \dots, w_{n}) \in \Delta_{n}$, let
\begin{displaymath}
\Phi(\mathbb{A}) = \sum_{i=1}^{n} w_{i} A_{ii}
\end{displaymath}
for an $n$-by-$n$ block matrix $\mathbb{A} = [ A_{ij} ]$. Then it is strictly positive and unital linear map. Assume that $0 < m I \leq A_{i} \leq M I$ for all $i = 1, \dots, n$, where $M, m > 0$ are some constants. Applying \cite[Proposition 2.7.8]{Bh} to $\Phi$ with $A_{1} \oplus A_{2} \oplus \cdots \oplus A_{n}$, we obtain the reverse inequality of arithmetic-harmonic weighted mean inequality:
\begin{displaymath}
\sum_{i=1}^{n} w_{i} A_{i} \leq \frac{(M + m)^{2}}{4 M m} \left[ \sum_{i=1}^{n} w_{i} A_{i}^{-1} \right]^{-1}.
\end{displaymath}
Here, the value $\displaystyle K = \frac{(M + m)^{2}}{4 M m}$ is known as the Kantorovich constant. By the $G$-harmonic weighted mean inequality in (P4),
\begin{equation} \label{E:reverse-A-G}
\sum_{i=1}^{n} w_{i} A_{i} \leq \frac{(M + m)^{2}}{4 M m} G_{n}(\omega; A_{1}, \dots, A_{n}).
\end{equation}

For each $n$ consider a multivariate geometric mean $G_{n}$ satisfying the consistency with scalars, that is,
\begin{displaymath}
G_{n}(\omega; \mathbf{A}) = \prod_{i=1}^{n} A_{i}^{w_{i}}
\end{displaymath}
when the $A_{i}$'s commute, where $\mathbf{A} = (A_{1}, \dots, A_{n}) \in \mathbb{P}^{n}$. Then the following holds:
\begin{displaymath}
G_{n}(\omega; G_{k}(\lambda; \mathbb{A}^{1}), \dots, G_{k}(\lambda; \mathbb{A}^{n})) = \prod_{i, j} A_{ij}^{w_{i} \lambda_{j}} = G_{k}(\lambda; G_{n}(\omega; \mathbb{A}_{1}), \dots, G_{n}(\omega; \mathbb{A}_{k}))
\end{displaymath}
when the $A_{ij}$'s commute, where $\lambda = (\lambda_{1}, \dots, \lambda_{k}) \in \Delta_{k}$ and $\mathbb{A} = [ A_{ij} ]$ is the $n$-by-$k$ block matrix with $A_{ij} \in \mathbb{P}$ for all $i,j$. Although it does not hold in general, we have the following inequality.

\begin{theorem} \label{T:Kantorovich1}
Let $\mathbb{A} = [ A_{ij} ]$ be the $n$-by-$k$ block matrix, where $A_{ij} \in \mathbb{P}$ for all $i,j$. Assume that $0 < m I \leq A_{ij} \leq M I$ for all $i, j$, where $M, m > 0$ are some constants. Then the ordered mean $G = \{ G_{n} \}_{n \in \mathbb{N}}$ satisfies that for any $\lambda \in \Delta_{k}$ and $\omega \in \Delta_{n}$
\begin{displaymath}
G_{n}(\omega; G_{k}(\lambda; \mathbb{A}^{1}), \dots, G_{k}(\lambda; \mathbb{A}^{n})) \leq K G_{k}(\lambda; G_{n}(\omega; \mathbb{A}_{1}), \dots, G_{n}(\omega; \mathbb{A}_{k})),
\end{displaymath}
where $\displaystyle K = \frac{(M + m)^{2}}{4 M m}$.
\end{theorem}

\begin{proof}
Let $\omega = (w_{1}, \dots, w_{n}) \in \Delta_{n}$. Then
\begin{displaymath}
\begin{split}
G_{n}(\omega; G_{k}(\lambda; \mathbb{A}^{1}), \dots, G_{k}(\lambda; \mathbb{A}^{n}))
& \leq \sum_{i=1}^{n} w_{i} G_{k}(\lambda; \mathbb{A}^{i}) \\
& \leq G_{k} \left( \lambda; \sum_{i=1}^{n} w_{i} \mathbb{A}^{i} \right) = G_{k} \left( \lambda; \sum_{i=1}^{n} w_{i} A_{i1}, \dots, \sum_{i=1}^{n} w_{i} A_{ik} \right) \\
& \leq G_{k} (\lambda; K G_{n}(\omega; \mathbb{A}_{1}), \dots, K G_{n}(\omega; \mathbb{A}_{k})) \\
& = K G_{k}(\lambda; G_{n}(\omega; \mathbb{A}_{1}), \dots, G_{n}(\omega; \mathbb{A}_{k})).
\end{split}
\end{displaymath}
The first inequality follows from the arithmetic-$G$ weighted mean inequality in (P4), the second from the joint concavity (P3), the third from \eqref{E:reverse-A-G} together with the monotonicity (P2), and the last equality from the homogeneity (P1).
\end{proof}

\begin{theorem} \label{T:Kantorovich2}
Let $\mathbb{A} = [ A_{ij} ]$ be the $n$-by-$k$ block matrix, where $A_{ij} \in \mathbb{P}$ for all $i,j$. Assume that $0 < m I \leq A_{ij} \leq M I$ for all $i, j$, where $M, m > 0$ are some constants. Then
\begin{displaymath}
G_{n}(\omega; G_{k}(\lambda; \mathbb{A}^{1}), \dots, G_{k}(\lambda; \mathbb{A}^{n})) \leq t G_{k}(\lambda; G_{n}(\omega; \mathbb{A}_{1}), \dots, G_{n}(\omega; \mathbb{A}_{k})) + \rho_{M, m} (t) I,
\end{displaymath}
where
\begin{displaymath}
\rho_{M, m} (t) =
\left\{
  \begin{array}{ll}
    (1 - t) m, & \hbox{$t \geq M/m$;} \\
    M + m - 2 \sqrt{t M m}, & \hbox{$m/M \leq t \leq M/m$;} \\
    (1 - t) M, & \hbox{$t \leq m/M$.}
  \end{array}
\right.
\end{displaymath}
\end{theorem}

\begin{proof}
It has been shown from \cite[Theorem 2.2]{KMS} that
\begin{displaymath}
\Psi(A) - t \Psi(A^{-1})^{-1} \leq \Psi(A) - t \left[ -\frac{1}{M m} \Psi(A) + \frac{M + m}{M m} I \right]^{-1}
\end{displaymath}
for any positive unital linear map $\Psi$ and any $t > 0$, where $A \in \mathbb{P}$ with $0 < m I \leq A \leq M I$. One can see easily that the function $\displaystyle f(x) = x - \frac{t M m}{M + m - x}$ has only one critical point $x_{0} = M + m - \sqrt{t M m}$ and $f''(x) < 0$ in the closed interval $[m, M]$. Thus, by fundamental calculation
\begin{displaymath}
\rho_{M, m} (t) = \underset{x \in [m, M]}{\max} f(x)
\end{displaymath}
is obtained as above, and
\begin{displaymath}
\Psi(A) \leq t \Psi(A^{-1})^{-1} + \rho_{M, m} (t) I
\end{displaymath}
for each $t > 0$ and any $A \in \mathbb{P}$ with $0 < m I \leq A \leq M I$. Then by the arithmetic-$G$-harmonic weighted mean inequalities in (P4),
\begin{displaymath}
\begin{split}
& G_{n}(\omega; G_{k}(\lambda; \mathbb{A}^{1}), \dots, G_{k}(\lambda; \mathbb{A}^{n})) - t G_{k}(\lambda; G_{n}(\omega; \mathbb{A}_{1}), \dots, G_{n}(\omega; \mathbb{A}_{k})) \\
& \leq \sum_{i,j} w_{i} \lambda_{j} A_{ij} - t \left[ \sum_{i,j} w_{i} \lambda_{j} A_{ij}^{-1} \right]^{-1}
= \Phi (\widehat{\mathbb{A}}) - t \Phi (\widehat{\mathbb{A}}^{-1})^{-1} \leq \rho_{M, m} (t) I,
\end{split}
\end{displaymath}
where $\displaystyle \Phi (\widehat{\mathbb{A}}) = \sum_{i,j} w_{i} \lambda_{j} A_{ij}$ for
$$ \widehat{\mathbb{A}} := A_{11} \oplus \cdots \oplus A_{1k} \oplus A_{21} \oplus \cdots \oplus A_{2k} \oplus \cdots \oplus A_{n1} \oplus \cdots \oplus A_{nk} $$ is the positive unital linear map for given probability vectors $\omega = (w_{1}, \dots, w_{n})$ and $\lambda = (\lambda_{1}, \dots, \lambda_{k})$.
\end{proof}

\begin{remark}
For $t = 1$ in Theorem \ref{T:Kantorovich2},
\begin{displaymath}
G_{n}(\omega; G_{k}(\lambda; \mathbb{A}^{1}), \dots, G_{k}(\lambda; \mathbb{A}^{n})) \leq G_{k}(\lambda; G_{n}(\omega; \mathbb{A}_{1}), \dots, G_{n}(\omega; \mathbb{A}_{k})) + (\sqrt{M} - \sqrt{m})^{2} I.
\end{displaymath}
\end{remark}

\section{Parameterized ordered means}

For given ordered mean $G = \{ G_{n} \}$, we define the parameterized ordered means $G^{\mu}: \Delta_{n} \times \mathbb{P}^{n} \to \mathbb{P}$ as
\begin{equation} \label{E:paramean}
G^{\mu}(\omega; \mathbf{A}) := G(\omega; \mathbf{A} + \mu \mathbf{I}) - \mu I, \ \mu \geq 0
\end{equation}
where $\mathbf{I} = (I, \dots, I) \in \mathbb{P}^{n}$ and $I$ is the identity operator, and
\begin{equation} \label{E:negative}
G^{\mu}(\omega; \mathbf{A}) := G^{-\mu}(\omega; \mathbf{A}^{-1})^{-1}, \ \ \mu < 0.
\end{equation}
We also denote as
\begin{center}
$\displaystyle G^{\infty}(\omega; \mathbf{A}) = \lim_{\mu \to \infty} G^{\mu}(\omega; \mathbf{A})$ \ and \ $\displaystyle G^{-\infty}(\omega; \mathbf{A}) = \lim_{\mu \to -\infty} G^{\mu}(\omega; \mathbf{A})$.
\end{center}
Using the arithmetic-$G$-harmonic weighted mean inequalities in (P4) and \eqref{E:negative}, we have
\begin{center}
$\displaystyle G^{\infty}(\omega; \mathbf{A}) = \mathcal{A}(\omega; \mathbf{A})$ \ and \ $\displaystyle G^{-\infty}(\omega; \mathbf{A}) = \mathcal{H}(\omega; \mathbf{A})$.
\end{center}

Y. Lim \cite{Lim} has established many remarkable properties of parameterized ordered means including a stochastic approximation and $L^{1}$ ergodic theorem for the parameterized Cartan (Karcher) mean.
From \cite[Proposition 5.3]{Lim}, we have the following properties of the parameterized ordered means $G^{\mu}$ induced from those of ordered means $G$.
\begin{proposition} \label{P:Paramean}
Let $\mathbf{A} = (A_{1}, \dots, A_{n}), \ \mathbf{B} = (B_{1}, \dots, B_{n}) \in \mathbb{P}^{n}, \ \omega = (w_{1}, \dots, w_{n}) \in \Delta_{n}$. The parameterized ordered mean $G^{\mu}$ for $\mu \in [-\infty, \infty]$ fulfills the following:
\begin{itemize}
\item[(1)] $($Homogeneity$)$ For a positive real number $a$ \\
$\left\{
  \begin{array}{ll}
    \displaystyle G^{\mu}(\omega; a \mathbf{A}) = a G^{\frac{\mu}{a}}(\omega; \mathbf{A}), & \hbox{$\mu \in [0, \infty]$;} \\
    \displaystyle G^{\mu}(\omega; a \mathbf{A}) = a G^{a \mu}(\omega; \mathbf{A}), & \hbox{$\mu \in [-\infty, 0)$;}
  \end{array}
\right.$ \vskip 1mm
\item[(2)] $($Monotonicity on variables$)$ If $B_{i} \leq A_{i}$ for all $1 \leq i \leq n,$ then
\begin{displaymath}
G^{\mu}(\omega; \mathbf{B}) \leq G^{\mu}(\omega; \mathbf{A});
\end{displaymath}
\item[(3)] $($Joint concavity$)$ For $\mu \in [0, \infty]$ and $0 \leq s \leq 1$
\begin{displaymath}
G^{\mu}(\omega; (1 - s) \mathbf{A} + s \mathbf{B}) \geq (1 - s) G^{\mu}(\omega; \mathbf{A}) + s G^{\mu}(\omega; \mathbf{B});
\end{displaymath}
\item[(4)] $($Arithmetic-$G^{\mu}$-harmonic weighted mean inequalities$)$
\begin{displaymath}
\left[ \sum_{i=1}^{n} w_{i} A_{i}^{-1} \right]^{-1} \leq G^{\mu}(\omega; \mathbf{A}) \leq \sum_{i=1}^{n} w_{i} A_{i};
\end{displaymath}
\item[(5)] $($Monotonicity on parameters$)$ For $0 \leq \nu \leq \mu \leq \infty$,
\begin{displaymath}
\mathcal{H} = G^{-\infty} \leq \cdots \leq G^{-\mu} \leq G^{-\nu} \leq \cdots \leq G^{0} = G \leq \cdots \leq G^{\nu} \leq G^{\mu} \leq \cdots \leq G^{\infty} = \mathcal{A};
\end{displaymath}
\item[(6)] $($Non-expansiveness$)$ $G^{\mu}$ is non-expansive for the Thompson metric, that is,
\begin{displaymath}
d_{T} (G^{\mu}(\omega; \mathbf{A}), G^{\mu}(\omega; \mathbf{B})) \leq \underset{1 \leq i \leq n}{\max} d_{T} (A_{i}, B_{i}).
\end{displaymath}
\end{itemize}
\end{proposition}

\begin{proof}
Most of properties have been proved in \cite[Proposition 5.3]{Lim}. Especially, the arithmetic-$G^{\mu}$-harmonic weighted mean inequalities (4) is derived from
\begin{displaymath}
\mathcal{H}(\omega; \mathbf{A}) \leq \mathcal{R}^{\mu}(\omega; \mathbf{A}) \leq G^{\mu}(\omega; \mathbf{A}) \leq \mathcal{A}(\omega; \mathbf{A}),
\end{displaymath}
where $\displaystyle \mathcal{R}^{\mu}(\omega; \mathbf{A}) := \left[ \sum_{i=1}^{n} w_{i} (A_{i} + \mu I)^{-1} \right]^{-1} - \mu I$ is the resolvent mean \cite{KLL}.

We show the homogeneity (1).
\begin{itemize}
\item[(1)] Let $a > 0$. For $\mu \geq 0$, by the homogeneity of ordered means (P1)
\begin{displaymath}
\begin{split}
G^{\mu}(\omega; a \mathbf{A}) & = G(\omega; a A_{1} + \mu I, \dots, a A_{n} + \mu I) - \mu I \\
& = a G(\omega; A_{1} + (\mu/a) I, \dots, A_{n} + (\mu/a) I) - \mu I = a G^{\frac{\mu}{a}}(\omega; \mathbf{A}).
\end{split}
\end{displaymath}
For $\mu < 0$, similarly by using the above result together with \eqref{E:negative}
\begin{displaymath}
G^{\mu}(\omega; a \mathbf{A}) = G^{-\mu} \left( \omega; a^{-1} \mathbf{A}^{-1} \right)^{-1}
= \left[ a^{-1} G^{-a \mu}(\omega; \mathbf{A}^{-1}) \right]^{-1} = a G^{a \mu}(\omega; \mathbf{A}).
\end{displaymath}
\end{itemize}
\end{proof}

\begin{proposition}
Let $G = \{ G_{n} \}$ be the ordered mean satisfying that for each $n$ and any $\omega = (w_{1}, \dots, w_{n}) \in \Delta_{n}$
\begin{itemize}
\item[(i)] $G_{n}$ is invariant under permutation, that is, for any permutation $\sigma$ on $n$ letters
\begin{displaymath}
G_{n}(\omega_{\sigma}; \mathbf{A}_{\sigma}) = G_{n} (\omega; \mathbf{A}),
\end{displaymath}
where $\omega_{\sigma} = (w_{\sigma(1)}, \dots, w_{\sigma(n)})$ and $\mathbf{A}_{\sigma} = (A_{\sigma(1)}, \dots, A_{\sigma(n)})$,
\item[(ii)] $G_{n}$ is invariant under repetition, that is, for each natural number $k \in \mathbb{N}$
\begin{displaymath}
G_{nk}(\omega^{(k)}; \underbrace{A_{1}, \dots, A_{n}}, \dots, \underbrace{A_{1}, \dots, A_{n}}) = G_{n} (\omega; A_{1}, \dots, A_{n}),
\end{displaymath}
where $\omega^{(k)} = \frac{1}{k} (\underbrace{w_{1}, \dots, w_{n}}, \dots, \underbrace{w_{1}, \dots, w_{n}}) \in \Delta_{nk}$,
\item[(iii)] $G_{n}$ is invariant under congruence transformation, that is, for any invertible operator $S \in B(\mathcal{H})$
\begin{displaymath}
G_{n}(\omega; S^{*} \mathbf{A} S) = S^{*} G_{n}(\omega; \mathbf{A}) S,
\end{displaymath}
where $S^{*} \mathbf{A} S = (S^{*} A_{1} S, \dots, S^{*} A_{n} S)$,
\item[(iv)] $G_{n} (\omega; A_{1}, \dots, A_{n-1}, X) = X$ if and only if $X = G_{n-1} (\hat{\omega}; A_{1}, \dots, A_{n-1})$, where $\hat{\omega} = \frac{1}{1 - w_{n}} (w_{1}, \dots, w_{n-1}) \in \Delta_{n-1}$,
\item[(v)] $\Phi(G_{n}(\omega; A_{1}, \dots, A_{n})) \leq G_{n}(\omega; \Phi(A_{1}), \dots, \Phi(A_{n}))$ for any positive unital linear map $\Phi$.
\end{itemize}
Then the corresponding parameterized ordered mean $G^{\mu}$ for $\mu \in [-\infty, \infty]$ holds the same properties (i), (ii), (iii) for any unitary operator $S$, and (iv). Moreover,
\begin{displaymath}
\left\{
  \begin{array}{ll}
    \displaystyle \Phi(G_{n}^{\mu}(\omega; A_{1}, \dots, A_{n})) \leq G_{n}^{\mu}(\omega; \Phi(A_{1}), \dots, \Phi(A_{n})), & \hbox{$\mu \in [0, \infty]$;} \\
    \displaystyle \Phi(G_{n}^{\mu}(\omega; A_{1}, \dots, A_{n})) \geq G_{n}^{\mu}(\omega; \Phi(A_{1}^{-1})^{-1}, \dots, \Phi(A_{n}^{-1})^{-1}), & \hbox{$\mu \in [-\infty, 0)$.}
  \end{array}
\right.
\end{displaymath}
\end{proposition}

\begin{proof}
It is obvious from the properties (i)-(iv) of the ordered means $G$ that the corresponding parameterized ordered mean $G^{\mu}$ for $\mu \in [-\infty, \infty]$ holds the same properties.

For $\mu \in [0, \infty]$, applying (v) with the positive unital linear map $\Phi$ we have
\begin{displaymath}
\begin{split}
\Phi(G_{n}^{\mu} (\omega; A_{1}, \dots, A_{n})) & \leq G_{n} (\omega; \Phi(A_{1} + \mu I), \dots, \Phi(A_{n} + \mu I)) - \mu I \\
& = G_{n} (\omega; \Phi(A_{1}) + \mu I, \dots, \Phi(A_{n}) + \mu I) - \mu I \\
& = G_{n}^{\mu}(\omega; \Phi(A_{1}), \dots, \Phi(A_{n})).
\end{split}
\end{displaymath}
Similarly, for $\mu \in [-\infty, 0)$
\begin{displaymath}
\begin{split}
\Phi(G_{n}^{\mu} (\omega; A_{1}, \dots, A_{n})) & = \Phi \left( G_{n}^{-\mu} (\omega; A_{1}^{-1}, \dots, A_{n}^{-1})^{-1} \right) \\
& \geq \Phi \left( G_{n}^{-\mu} (\omega; A_{1}^{-1}, \dots, A_{n}^{-1}) \right)^{-1} \\
& \geq G_{n}^{-\mu} \left( \omega; \Phi(A_{1}^{-1}), \dots, \Phi(A_{n}^{-1}) \right)^{-1} \\
& = G_{n}^{\mu}(\omega; \Phi(A_{1}^{-1})^{-1}, \dots, \Phi(A_{n}^{-1})^{-1}).
\end{split}
\end{displaymath}
The first and last equalities follow from the definition \eqref{E:negative}, the first inequality from Choi's inequality in \cite[Theorem 2.3.6]{Bh}, and the second inequality from (v) and the order reversing of inversion.
\end{proof}

\begin{theorem} \label{T:Paramean-concave}
Let $\mathbb{A} = [ A_{ij} ]$ be the $n$-by-$k$ block matrix, where $A_{ij} \in \mathbb{P}$ for all $i,j$. Let $\omega = (w_{1}, \dots, w_{n}) \in \Delta_{n}$ and $\lambda = (\lambda_{1}, \dots, \lambda_{k}) \in \Delta_{k}$. Then for any $\mu_{1}, \dots, \mu_{n} \geq 0$
\begin{displaymath}
\sum_{i=1}^{n} w_{i} G_{k}^{\mu_{i}}(\lambda; \mathbb{A}^{i}) \leq G_{k}^{\omega \bullet \mu} \left( \lambda; \sum_{i=1}^{n} w_{i} \mathbb{A}^{i} \right),
\end{displaymath}
where $\displaystyle \omega \bullet \mu = \sum_{i=1}^{n} w_{i} \mu_{i}$ for $\mu := (\mu_{1}, \dots, \mu_{n})$.
\end{theorem}

\begin{proof}
By the joint concavity of ordered means (P3), we have
\begin{displaymath}
\begin{split}
\sum_{i=1}^{n} w_{i} G_{k}^{\mu_{i}}(\lambda; \mathbb{A}^{i}) & = \sum_{i=1}^{n} w_{i} G_{k} (\lambda; \mathbb{A}^{i} + \mu_{i} I) - \sum_{i=1}^{n} w_{i} \mu_{i} I \\
& \leq G_{k} \left( \lambda; \sum_{i=1}^{n} w_{i} (\mathbb{A}^{i} + \mu_{i} I) \right) - \sum_{i=1}^{n} w_{i} \mu_{i} I = G_{k}^{\omega \bullet \mu} \left( \lambda; \sum_{i=1}^{n} w_{i} \mathbb{A}^{i} \right).
\end{split}
\end{displaymath}
\end{proof}

\begin{remark}
For $n = 2$, taking $\mu_{1} = \mu_{2} = \nu (\geq 0)$ and $\omega = (1 - t, t)$ for $t \in [0,1]$ in Theorem \ref{T:Paramean-concave} yields the joint concavity in Proposition \ref{P:Paramean} (3):
\begin{displaymath}
(1-t) G^{\nu}(\lambda; \mathbf{A}) + t G^{\nu}(\lambda; \mathbf{B}) \leq G^{\nu}(\lambda; (1-t) \mathbf{A} + t \mathbf{B}).
\end{displaymath}
So Theorem \ref{T:Paramean-concave} is a multivariate extension of the joint concavity.
\end{remark}

\begin{theorem}
Let $\mathbb{A} = [ A_{ij} ]$ be the $n$-by-$k$ block matrix, where $0 < m I \leq A_{ij} \leq M I$ for some constants $M, m > 0$. Let $\omega = (w_{1}, \dots, w_{n}) \in \Delta_{n}$ and $\lambda = (\lambda_{1}, \dots, \lambda_{k}) \in \Delta_{k}$. Then
\begin{itemize}
\item[(i)] for any $\mu_{1}, \dots, \mu_{n}, \nu \geq 0$
\begin{displaymath}
G_{n}^{\nu}(\omega; G_{k}^{\mu_{1}}(\lambda; \mathbb{A}^{1}), \dots, G_{k}^{\mu_{n}}(\lambda; \mathbb{A}^{n})) \leq K G_{k}^{\omega \bullet \mu} (\lambda; G_{n}^{\nu}(\omega; \mathbb{A}_{1}), \dots, G_{n}^{\nu}(\omega; \mathbb{A}_{k}));
\end{displaymath}
\item[(ii)] for any $\mu_{1}, \dots, \mu_{n}, \nu < 0$
\begin{displaymath}
G_{n}^{\nu}(\omega; G_{k}^{\mu_{1}}(\lambda; \mathbb{A}^{1}), \dots, G_{k}^{\mu_{n}}(\lambda; \mathbb{A}^{n})) \leq K^{-1} G_{k}^{\omega \bullet \mu} (\lambda; G_{n}^{\nu}(\omega; \mathbb{A}_{1}), \dots, G_{n}^{\nu}(\omega; \mathbb{A}_{k})),
\end{displaymath}
\end{itemize}
where $\displaystyle K = \frac{(M + m)^{2}}{4 M m}$.
\end{theorem}

\begin{proof}
Assume that $0 < m I \leq A_{ij} \leq M I$ for some constants $M, m > 0$, where $\mathbb{A} = [ A_{ij} ]$ is the $n$-by-$k$ block matrix.
\begin{itemize}
\item[(i)] Since the parameterized ordered mean $G^{\nu}$ satisfies the arithmetic-$G^{\nu}$-harmonic weighted mean inequalities in Proposition \ref{P:Paramean} (4), we have from \eqref{E:reverse-A-G}
\begin{equation} \label{E:reverse-A-paraG}
\sum_{i=1}^{n} w_{i} A_{i} \leq K G_{n}^{\nu}(\omega; A_{1}, \dots, A_{n}).
\end{equation}
Then for $\mu_{1}, \dots, \mu_{n}, \nu \geq 0$
\begin{displaymath}
\begin{split}
G_{n}^{\nu}(\omega; G_{k}^{\mu_{1}}(\lambda; \mathbb{A}^{1}), \dots, G_{k}^{\mu_{n}}(\lambda; \mathbb{A}^{n})) & \leq \sum_{i=1}^{n} w_{i} G_{k}^{\mu_{i}}(\lambda; \mathbb{A}^{i}) \\
& \leq G_{k}^{\omega \bullet \mu} \left( \lambda; \sum_{i=1}^{n} w_{i} \mathbb{A}^{i} \right) \\
& = G_{k}^{\omega \bullet \mu} \left( \lambda; \sum_{i=1}^{n} w_{i} A_{i1}, \dots, \sum_{i=1}^{n} w_{i} A_{ik} \right) \\
& \leq G_{k}^{\omega \bullet \mu} (\lambda; K G_{n}^{\nu}(\omega; \mathbb{A}_{1}), \dots, K G_{n}^{\nu}(\omega; \mathbb{A}_{k})) \\
& = K G_{k}^{\frac{\omega \bullet \mu}{K}} (\lambda; G_{n}^{\nu}(\omega; \mathbb{A}_{1}), \dots, G_{n}^{\nu}(\omega; \mathbb{A}_{k})) \\
& \leq K G_{k}^{\omega \bullet \mu} (\lambda; G_{n}^{\nu}(\omega; \mathbb{A}_{1}), \dots, G_{n}^{\nu}(\omega; \mathbb{A}_{k})).
\end{split}
\end{displaymath}
The first inequality follows from the arithmetic-$G^{\nu}$ weighted mean inequality in Proposition \ref{P:Paramean} (4), the second inequality from Theorem \ref{T:Paramean-concave}, the third inequality from \eqref{E:reverse-A-paraG}, the second equality from the homogeneity in Proposition \ref{P:Paramean} (1), and the last inequality from the monotonicity of parameterized ordered means for parameters in Proposition \ref{P:Paramean} (5) since $K \geq 1$.
\item[(ii)] For $\mu_{1}, \dots, \mu_{n}, \nu < 0$
\begin{displaymath}
\begin{split}
G_{n}^{\nu}(\omega; G_{k}^{\mu_{1}}(\lambda; \mathbb{A}^{1}), \dots, G_{k}^{\mu_{n}}(\lambda; \mathbb{A}^{n})) & = G_{n}^{-\nu} \left( \omega; G_{k}^{-\mu_{1}} \left( \lambda; (\mathbb{A}^{1})^{-1} \right), \dots, G_{k}^{-\mu_{n}} \left( \lambda; (\mathbb{A}^{n})^{-1} \right) \right)^{-1} \\
& \geq K^{-1} G_{k}^{\omega \bullet (-\mu)} \left( \lambda; G_{n}^{-\nu} (\omega; (\mathbb{A}_{1})^{-1}), G_{n}^{-\nu} (\omega; (\mathbb{A}_{k})^{-1}) \right)^{-1} \\
& = K^{-1} G_{k}^{- \omega \bullet \mu} \left( \lambda; G_{n}^{\nu} (\omega; \mathbb{A}_{1})^{-1}, G_{n}^{\nu} (\omega; \mathbb{A}_{k})^{-1} \right)^{-1} \\
& = K^{-1} G_{k}^{\omega \bullet \mu} (\lambda; G_{n}^{\nu}(\omega; \mathbb{A}_{1}), \dots, G_{n}^{\nu}(\omega; \mathbb{A}_{k})).
\end{split}
\end{displaymath}
The first equality follows from \eqref{E:negative}, the inequality from (i) and the order reversing of inversion, and the second and last equalities again from \eqref{E:negative}.
\end{itemize}
\end{proof}

\section{Interpolation with power means}

Let $\mathbf{a} = (a_{1}, a_{2}, \dots, a_{n})$ be an $n$-tuple of given positive real numbers and a positive probability vector $\omega = (w_{1}, \dots, w_{n})$. The generalized means, also called the H\"{o}lder mean, with exponent $p \in \mathbb{R}$ are a family of functions defined as
\begin{displaymath}
\mathfrak{M}_{p} (\omega; \mathbf{a}) := \left( \sum_{i=1}^{n} w_{i} a_{i}^{p} \right)^{\frac{1}{p}}, \ p \neq 0,
\end{displaymath}
and
\begin{displaymath}
\mathfrak{M}_{0} (\omega; \mathbf{a}) := \lim_{p \to 0} \mathfrak{M}_{p} (\omega; \mathbf{a}) = \prod_{i=1}^{n} a_{i}^{w_{i}}.
\end{displaymath}
Note that
\begin{displaymath}
\begin{split}
\mathfrak{M}_{\infty} (\omega; \mathbf{a}) & := \lim_{p \to \infty} \mathfrak{M}_{p} (\omega; \mathbf{a}) = \max \{ a_{1}, \dots, a_{n} \}, \\
\mathfrak{M}_{-\infty} (\omega; \mathbf{a}) & := \lim_{p \to -\infty} \mathfrak{M}_{p} (\omega; \mathbf{a}) = \min \{ a_{1}, \dots, a_{n} \}.
\end{split}
\end{displaymath}
One of the interesting properties for the generalized means is the monotonicity for exponents, that is,
\begin{equation} \label{E:M-mono}
\mathfrak{M}_{p} (\omega; \mathbf{a}) \leq \mathfrak{M}_{q} (\omega; \mathbf{a}) \ \ \textrm{if} \ p \leq q.
\end{equation}

Via the theory of power means of positive definite Hermitian matrices in \cite{LP}, the power means of positive invertible operators have been successfully defined and developed in \cite{LL14}. The power mean $P_{p}(\omega; \mathbf{A})$ of $\mathbf{A} = (A_{1}, \dots, A_{n}) \in \mathbb{P}^{n}$ for $p \in (0,1]$ is the unique solution $X \in \mathbb{P}$ of the nonlinear equation
\begin{displaymath}
X = \sum_{i=1}^{n} w_{i} X \#_{p} A_{i},
\end{displaymath}
and $P_{p}(\omega; \mathbf{A}) = P_{-p}(\omega; \mathbf{A}^{-1})^{-1}$ for $p \in [-1,0)$. It is the operator version of generalized mean $\mathfrak{M}_{p}$ of positive scalars, in other words, $P_{p}(\omega; \mathbf{A}) = \mathfrak{M}_{p} (\omega; \mathbf{A})$ if the $A_{i}$'s commute. The most interesting results shown in \cite{LL14} are that the power means converges to the Karcher mean under the strong operator topology, i.e.
\begin{displaymath}
\lim_{p \to 0} P_{p}(\omega; \mathbf{A}) = \Lambda(\omega; \mathbf{A}),
\end{displaymath}
where the Karcher mean $\Lambda(\omega; \mathbf{A})$ is the unique solution $X \in \mathbb{P}$ of the Karcher equation $\sum_{i=1}^{n} w_{i} \log X^{1/2} A_{i}^{-1} X^{1/2} = 0$, and for $0 < p \leq q \leq 1$
\begin{equation} \label{E:power}
\mathcal{H} = P_{-1} \leq \cdots \leq P_{-q} \leq P_{-p} \leq \cdots \leq P_{0} = \Lambda \leq \cdots \leq P_{p} \leq P_{q} \leq \cdots \leq P_{1} = \mathcal{A}.
\end{equation}

For two parameters $\mu, \nu > 0$ and $t \in [0,1]$, the generalized mean $\mathfrak{M}_{p} := \mathfrak{M}_{p} (1-t, t; \mu, \nu)$ with exponent $p \in \mathbb{R}$ have the following chain by \eqref{E:M-mono}
\begin{displaymath}
\mathfrak{M}_{-\infty} \leq \mathfrak{M}_{-q} \leq \mathfrak{M}_{-p} \leq \mathfrak{M}_{p} \leq \mathfrak{M}_{q} \leq \mathfrak{M}_{\infty}
\end{displaymath}
for $0 \leq p \leq q$. Thus, Proposition \ref{P:Paramean} (5) provides the the following chain of parameterized ordered means $G^{\mathfrak{M}_{p}} := G^{\mathfrak{M}_{p} (1-t, t; \mu, \nu)} (\omega; \mathbf{A})$ for $p \in \mathbb{R}$
\begin{equation} \label{E:paramean-mono}
G^{\mathfrak{M}_{-\infty}} \leq G^{\mathfrak{M}_{-q}} \leq G^{\mathfrak{M}_{-p}} \leq G^{\mathfrak{M}_{p}} \leq G^{\mathfrak{M}_{q}} \leq G^{\mathfrak{M}_{\infty}}.
\end{equation}

For fixed $p \in [-1,1]$ and $\mu, \nu > 0$, we consider the following two families of parameterized ordered means \begin{displaymath}
\{ G^{P_{p} (1-t, t; \mu, \nu)} (\omega; \mathbf{A}) \}_{t \in [0,1]}, \ \{ P_{p} (1-t, t; G^{\mu} (\omega; \mathbf{A}), G^{\nu} (\omega; \mathbf{A})) \}_{t \in [0,1]}.
\end{displaymath}
These are continuous curves in $\mathbb{P}$ connecting two parameterized ordered means $G^{\mu} (\omega; \mathbf{A})$ at $t = 0$ and $G^{\nu} (\omega; \mathbf{A})$ at $t = 1$.
\begin{remark}
Two families $\{ G^{P_{p} (1-t, t; \mu, \nu)} \}_{t \in [0,1]}, \ \{ P_{p} (1-t, t; G^{\mu}, G^{\nu} \}_{t \in [0,1]}$ are interpolating monotonically two parameterized ordered means $G^{\mu} (\omega; \mathbf{A})$ and $G^{\nu} (\omega; \mathbf{A})$ depending on the size of $\mu$ and $\nu$. Indeed, without loss of generality, assume $0 < \mu \leq \nu$. Then the generalized mean with fixed exponent $p \in [-1,1]$
\begin{displaymath}
P_{p} (1-t, t; \mu, \nu) =
\left\{
  \begin{array}{ll}
    \mu \left[ 1-t + t x^{p} \right]^{\frac{1}{p}}, & \hbox{$p \neq 0$;} \\
    \mu x^{t}, & \hbox{$p = 0$.}
  \end{array}
\right.
\end{displaymath}
for $x = \frac{\nu}{\mu} \geq 1$ is an increasing function on $t \in [0,1]$. So the family $\{ G^{P_{p} (1-t, t; \mu, \nu)} \}_{t \in [0,1]}$ is monotonically increasing on $t$ by the monotonicity of parameterized ordered means on parameters in Proposition \ref{P:Paramean} (5). Moreover, $G^{\mu} \leq G^{\nu}$ again by Proposition \ref{P:Paramean} (5), and
\begin{displaymath}
P_{p} (1-t, t; G^{\mu}, G^{\nu}) =
\left\{
  \begin{array}{ll}
    G^{\mu} \#_{\frac{1}{p}} \left[ (1-t) G^{\mu} + t G^{\mu} \#_{p} G^{\nu} \right], & \hbox{$p \in (0,1]$;} \\
    G^{\mu} \#_{t} G^{\nu}, & \hbox{$p = 0$;} \\
    G^{\mu} \#_{-\frac{1}{p}} \left[ (1-t) (G^{\mu})^{-1} + t (G^{\mu})^{-1} \#_{p} (G^{\nu})^{-1} \right]^{-1}, & \hbox{$p \in [-1,0)$.}
  \end{array}
\right.
\end{displaymath}
by \cite[Proposition 3.8]{LP}. So the family $\{ P_{p} (1-t, t; G^{\mu}, G^{\nu} \}_{t \in [0,1]}$ is also monotonically increasing on $t$.
\end{remark}

The following shows the relation between the above families of parameterized ordered means for $p = 1$.
\begin{theorem} \label{T:comparison}
Let $\mathbf{A} = (A_{1}, \dots, A_{n}) \in \mathbb{P}^{n}$ and $\omega = (w_{1}, \dots, w_{n}) \in \Delta_{n}$. Then for all $t \in [0,1]$
\begin{itemize}
\item[(i)] $(1-t) G^{\mu}(\omega; \mathbf{A}) + t G^{\nu}(\omega; \mathbf{A}) \leq G^{(1-t) \mu + t \nu}(\omega; \mathbf{A})$ for any $\mu, \nu \geq 0$, and
\item[(ii)] $(1-t) G^{\mu}(\omega; \mathbf{A}) + t G^{\nu}(\omega; \mathbf{A}) \geq G^{(1-t) \mu + t \nu}(\omega; \mathbf{A})$ for any $\mu, \nu < 0$.
\end{itemize}
\end{theorem}

\begin{proof}
Taking the $n$-by-$n$ block matrix
\begin{displaymath}
\mathbb{A} =
\left(
  \begin{array}{cccc}
    A_{1} & A_{2} & \cdots & A_{n} \\
    A_{1} & A_{2} & \cdots & A_{n} \\
    \vdots & \vdots & \ddots & \vdots \\
    A_{1} & A_{2} & \cdots & A_{n} \\
  \end{array}
\right)
\end{displaymath}
in Theorem \ref{T:Paramean-concave} and using the arithmetic-$G$ weighted mean inequality, we obtain (i).

For any $\mu, \nu < 0$,
\begin{displaymath}
\begin{split}
(1-t) G^{\mu}(\omega; \mathbf{A}) + t G^{\nu}(\omega; \mathbf{A}) & = (1-t) G^{-\mu}(\omega; \mathbf{A}^{-1})^{-1} + t G^{-\nu}(\omega; \mathbf{A}^{-1})^{-1} \\
& \geq \left[ (1-t) G^{-\mu}(\omega; \mathbf{A}^{-1}) + t G^{-\nu}(\omega; \mathbf{A}^{-1}) \right]^{-1} \\
& \geq G^{- (1-t) \mu - t \nu}(\omega; \mathbf{A}^{-1})^{-1} = G^{(1-t) \mu + t \nu}(\omega; \mathbf{A}).
\end{split}
\end{displaymath}
The first equality follows from \eqref{E:negative}, the first inequality from the convexity of inversion, and the second inequality from (i) and the order reversing of inversion.
\end{proof}

\begin{theorem} \label{T:power1}
Let $\mathbf{A} \in \mathbb{P}^{n}$, $\omega \in \Delta_{n}$ and $\mu, \nu > 0$. Then for all $t \in [0,1]$ and $p \in [-1,1)$
\begin{displaymath}
G^{P_{p} (1-t, t; \mu, \nu)} (\omega; \mathbf{A}) \geq P_{p} (1-t, t; G^{\frac{\mu}{K}} (\omega; \mathbf{A}), G^{\frac{\nu}{K}} (\omega; \mathbf{A})),
\end{displaymath}
where $\displaystyle K = \frac{(\mu + \nu)^{2}}{4 \mu \nu}$.
\end{theorem}

\begin{proof}
For two parameters $\mu, \nu > 0$,
\begin{displaymath}
\begin{split}
G^{P_{p} (1-t, t; \mu, \nu)} (\omega; \mathbf{A}) & \geq G^{\frac{(1-t) \mu + t \nu}{K}} (\omega; \mathbf{A}) = \frac{1}{K} G^{(1-t) \mu + t \nu} (\omega; K \mathbf{A}) \\
& \geq \frac{1}{K} \left[ (1-t) G^{\mu} (\omega; K \mathbf{A}) + t G^{\nu} (\omega; K \mathbf{A}) \right] \\
& \geq \frac{1}{K} P_{p} (1-t, t; G^{\mu} (\omega; K \mathbf{A}), G^{\nu} (\omega; K \mathbf{A})) \\
& = P_{p} (1-t, t; G^{\frac{\mu}{K}} (\omega; \mathbf{A}), G^{\frac{\nu}{K}} (\omega; \mathbf{A})).
\end{split}
\end{displaymath}
The first inequality follows from Proposition \ref{P:Paramean} (5) with
\begin{displaymath}
(1-t) \mu + t \nu \leq K [ (1-t) \mu^{-1} + t \nu^{-1} ]^{-1} \leq K P_{p} (1-t, t; \mu, \nu)
\end{displaymath}
for $t \in [0,1]$, the first equality from the homogeneity in Proposition \ref{P:Paramean} (1), the second inequality from Theorem \ref{T:comparison} (i), the third inequality from the arithmetic-power mean inequality in \eqref{E:power}, and the last equality from the homogeneities of power mean and parameterized ordered mean, respectively, in \cite[Proposition 3.6]{LL14} and Proposition \ref{P:Paramean} (1).
\end{proof}

\begin{remark}
Theorem \ref{T:power1} shows the order relation between two families
\begin{center}
$\{ G^{P_{p} (1-t, t; \mu, \nu)} (\omega; \mathbf{A}) \}_{t \in [0,1]}$ \ and \ $\{ P_{p} (1-t, t; G^{\frac{\mu}{K}} (\omega; \mathbf{A}), G^{\frac{\nu}{K}} (\omega; \mathbf{A})) \}_{t \in [0,1]}$
\end{center}
for $p \in [-1,1)$. Note that
\begin{displaymath}
P_{p} (1-t, t; G^{\mu} (\omega; \mathbf{A}), G^{\nu} (\omega; \mathbf{A})) \geq P_{p} (1-t, t; G^{\frac{\mu}{K}} (\omega; \mathbf{A}), G^{\frac{\nu}{K}} (\omega; \mathbf{A})),
\end{displaymath}
since $G^{\mu} \geq G^{\frac{\mu}{K}}$ and $G^{\nu} \geq G^{\frac{\nu}{K}}$ by the monotonicity of parameterized ordered means for parameters in Proposition \ref{P:Paramean} (5) and the power mean $P_{p}$ is monotonic on variables. So it is open to compare $\{ G^{P_{p} (1-t, t; \mu, \nu)} (\omega; \mathbf{A}) \}_{t \in [0,1]}$ and $\{ P_{p} (1-t, t; G^{\mu} (\omega; \mathbf{A}), G^{\nu} (\omega; \mathbf{A})) \}_{t \in [0,1]}$.
\end{remark}

Theorem \ref{T:power1} can be extended to the multivariable power means, and its proof follows similarly to the proof of Theorem \ref{T:power1} for the multivariable power means.
\begin{proposition}
Let $\mathbf{A} \in \mathbb{P}^{n}$, $\omega \in \Delta_{n}$, and $\lambda \in \Delta_{k}$. Then for all $\mu_{1}, \dots, \mu_{k} > 0$ and $p \in [-1,1)$
\begin{displaymath}
G^{P_{p} (\lambda; \mu_{1}, \dots, \mu_{k})} (\omega; \mathbf{A}) \geq P_{p} (\lambda; G^{\frac{\mu_{1}}{K}} (\omega; \mathbf{A}), \dots, G^{\frac{\mu_{k}}{K}} (\omega; \mathbf{A})),
\end{displaymath}
where $\displaystyle K = \frac{(\mu_{\max} + \mu_{\min})^{2}}{4 \mu_{\max} \mu_{\min}}$ for $\mu_{\max} = \max \{ \mu_{1}, \dots, \mu_{k} \}$ and $\mu_{\min} = \min \{ \mu_{1}, \dots, \mu_{k} \}$.
\end{proposition}

The following shows the relation between two families of parameterized ordered means associated with power means for parameters $p$ and $q$ with $-1 \leq p \leq 1 \leq q$.
\begin{theorem} \label{T:power2}
Let $\mathbf{A} \in \mathbb{P}^{n}$, $\omega \in \Delta_{n}$, and $\lambda = (\lambda_{1}, \dots, \lambda_{k}) \in \Delta_{k}$. Then for all $\mu_{1}, \dots, \mu_{k} > 0$ and $-1 \leq p \leq 1 \leq q$,
\begin{displaymath}
G^{P_{q} (\lambda; \mu_{1}, \dots, \mu_{k})} (\omega; \mathbf{A}) \geq P_{p} (\lambda; G^{\mu_{1}} (\omega; \mathbf{A}), \dots, G^{\mu_{k}} (\omega; \mathbf{A})).
\end{displaymath}
\end{theorem}

\begin{proof}
Let $\mu = (\mu_{1}, \dots, \mu_{k}) \in \mathbb{R}^{k}$ with positive components. Note that $P_{q} (\lambda; \mu_{1}, \dots, \mu_{k}) = \mathfrak{M}_{q} (\lambda; \mu_{1}, \dots, \mu_{k})$ for any $q \geq 1$. Then
\begin{displaymath}
G^{P_{q} (\lambda; \mu)} (\omega; \mathbf{A}) \geq G^{\lambda \bullet \mu} (\omega; \mathbf{A})
\geq \sum_{i=1}^{k} \lambda_{i} G^{\mu_{i}} (\omega; \mathbf{A})
\geq P_{p} (\lambda; G^{\mu_{1}} (\omega; \mathbf{A}), \dots, G^{\mu_{k}} (\omega; \mathbf{A})),
\end{displaymath}
where $\lambda \bullet \mu$ denotes the Euclidean inner product of $\lambda$ and $\mu$, alternatively the weighted arithmetic mean of $\mu$ with probability vector $\lambda$. The first inequality follows from the monotonicity on parameters in Proposition \ref{P:Paramean} (5) with \eqref{E:M-mono}, the second from Theorem \ref{T:Paramean-concave}, and the third from the monotonicity of power means in \eqref{E:power}.
\end{proof}

\begin{remark}
By \eqref{E:power}, Theorem \ref{T:power2}, and \eqref{E:paramean-mono}, we obtain a new chain of inequalities
\begin{displaymath}
\Lambda(\lambda; G^{\mu_{1}}, \dots, G^{\mu_{k}}) \leq P_{\frac{1}{q}}(\lambda; G^{\mu_{1}}, \dots, G^{\mu_{k}}) \leq P_{\frac{1}{p}}(\lambda; G^{\mu_{1}}, \dots, G^{\mu_{k}}) \leq G^{P_{p}(\lambda; \mu)} \leq G^{P_{q}(\lambda; \mu)}
\end{displaymath}
for $1 \leq p \leq q$, where $G^{\nu} := G^{\nu}(\omega; \mathbf{A})$ for any parameter $\nu \geq 0$. For negative parameters the reverse inequalities in the above chain hold, but it remains to show
\begin{displaymath}
P_{-\frac{1}{p}}(\lambda; G^{\mu_{1}}, \dots, G^{\mu_{k}}) \geq G^{P_{-p}(\lambda; \mu)}
\end{displaymath}
for $p \geq 1$.
\end{remark}

\section{Final remarks and open problems}

The generalized mean $\mathfrak{M}_{p}$ of positive real numbers can be naturally defined for positive definite operators $A_{1}, \dots, A_{n}$ such as
\begin{displaymath}
\mathfrak{M}_{p} (\omega; \mathbf{A}) := \left( \sum_{i=1}^{n} w_{i} A_{i}^{p} \right)^{\frac{1}{p}}, \ p \neq 0.
\end{displaymath}
It holds many properties same for positive real numbers, but there are some different ones. For instance,
\begin{displaymath}
\mathfrak{M}_{0} (\omega; \mathbf{A}) := \lim_{p \to 0} \mathfrak{M}_{p} (\omega; \mathbf{A}) = \exp \left( \sum_{i=1}^{n} w_{i} \log A_{i} \right),
\end{displaymath}
where the right hand side is known as the log-Euclidean mean, and
\begin{displaymath}
\mathfrak{M}_{-q} \leq \mathfrak{M}_{-p} \leq \mathfrak{M}_{-1} = \mathcal{H} \leq \mathfrak{M}_{1} = \mathcal{A} \leq \mathfrak{M}_{p} \leq \mathfrak{M}_{q}
\end{displaymath}
if $1 \leq p \leq q$. One can find more information from \cite{Kim18} by taking the finitely supported measure $\sum_{i=1}^{n} w_{i} \delta_{A_{i}}$, where $\delta_{A}$ is the point measure at $A \in \mathbb{P}$.

Similarly to Section 4, we can consider the following two families of parameterized ordered means
\begin{displaymath}
\{ G^{\mathfrak{M}_{p} (1-t, t; \mu, \nu)} (\omega; \mathbf{A}) \}, \ \{ \mathfrak{M}_{p} (1-t, t; G^{\mu} (\omega; \mathbf{A}), G^{\nu} (\omega; \mathbf{A})) \}
\end{displaymath}
for fixed $p \in \mathbb{R}$ and $\mu, \nu > 0$. By Theorem \ref{T:comparison} (i) and Theorem \ref{T:power1} with $p = -1$ we obtain
\begin{displaymath}
\begin{split}
G^{\mathcal{A}(1-t, t; \mu, \nu)}(\omega; \mathbf{A}) & \geq \mathcal{A}(1-t, t; G^{\mu}(\omega; \mathbf{A}), G^{\nu}(\omega; \mathbf{A})) \\
G^{\mathcal{H}(1-t, t; \mu, \nu)} (\omega; \mathbf{A}) & \geq \mathcal{H}(1-t, t; G^{\frac{\mu}{K}} (\omega; \mathbf{A}), G^{\frac{\nu}{K}} (\omega; \mathbf{A})).
\end{split}
\end{displaymath}
Note that it may not hold for $p \in (-1,1)$, since the monotonicity of generalized means including the log-Euclidean mean does not hold. So it is an interesting question to compare two families for general $p \in \mathbb{R}$.

At last, we explain the background to extend a multivariate mean to a barycenter of probability measures and give some open problems from the results in this paper. Let $B(X)$ be the algebra of Borel sets on a metric space $(X, d)$. Let $\mathcal{P}(X)$ be the set of all probability measures on $(X, B(X))$ with separable support, and let $\mathcal{P}^{p}(X) \subset \mathcal{P}(X)$ for $p \geq 1$ be the set of all probability measures with finite $p$-moment: for some $y \in X$
\begin{displaymath}
\int_{X} d^{p}(x,y) d\mu(x) < \infty.
\end{displaymath}
We denote as $\mathcal{P}^{\infty}(X)$ the set of probability measures on $(X, B(X))$ with compact support. For $p \geq 1$ the $p$-Wasserstein distance on $\mathcal{P}^{p}(X)$ is defined by
\begin{displaymath}
d_{p}^{W} (\rho, \sigma) := \left[ \underset{\pi \in \Pi(\rho, \sigma)}{\inf} \int_{X \times X} d^{p}(x,y) d\pi(x, y) \right]^{1/p},
\end{displaymath}
where $\Pi(\rho, \sigma)$ denotes a set of all couplings for $\rho, \sigma \in \mathcal{P}^{p}(X)$. Moreover, the $\infty$-Wasserstein distance on $\mathcal{P}^{\infty}(X)$ is given by
\begin{displaymath}
d_{\infty}^{W} (\rho, \sigma) = \lim_{p \to \infty} d_{p}^{W} (\rho, \sigma) = \underset{\pi \in \Pi(\rho, \sigma)}{\inf} \sup \{ d(x,y) : (x,y) \in \textrm{supp}(\pi) \}.
\end{displaymath}
For more details and information, see \cite{Vi03,Vi08}.

For each natural number $n$, in general, a mean $G_{n}$ on a set $X$ is a map $G_{n}: X^{n} \to X$ satisfying the idempotency. An intrinsic mean $G_{n}$ is the mean with invariance under permutation and repetition. From Proposition 2.7 in \cite{LL17}, a non-expansive intrinsic mean $G = \{ G_{n} \}_{n \in \mathbb{N}}$ on a complete metric space $X$ uniquely extends to a $d_{\infty}^{W}$-contractive barycentric map $\beta_{G}: \mathcal{P}^{\infty}(X) \to X$, where $\beta_{G}$ is $d_{\infty}^{W}$-contractive if and only if
\begin{displaymath}
d(\beta_{G}(\rho), \beta_{G}(\sigma)) \leq d_{\infty}^{W} (\rho, \sigma)
\end{displaymath}
for any $\rho, \sigma \in \mathcal{P}^{\infty}(X)$. Thus, the parameterized ordered mean $G^{\mu} = \{ G_{n}^{\mu} \}$ with invariance under permutation and repetition can be extended to a $d_{\infty}^{W}$-contractive barycentric map $\beta_{G^{\mu}}$ by Proposition \ref{P:Paramean} (6). It is also an interesting problem to generalize results in Section 3 and Section 4 to the $d_{\infty}^{W}$-contractive barycentric map $\beta_{G^{\mu}}$.

\vskip 4mm

\textbf{Acknowledgement}

This work was supported by the National Research Foundation of Korea (NRF) grant funded by the Korea government (MSIT) (No. NRF-2018R1C1B6001394).

\end{document}